\documentclass[12pt]{article}
\usepackage{amsfonts,amssymb,amsmath,amsthm}
\date{}
\newcommand{\R}{{\mathbb R}}
\newcommand{\Q}{{\mathbb Q}}
\newcommand{\Z}{{\mathbb Z}}
\newcommand{\N}{{\mathbb N}}
\newcommand{\T}{{\mathbb T}}
\newcommand{\Cl}{\mathrm{Cl}\,}
\newcommand{\const}{{\rm const}}
\newcommand{\pr}{\mathop{\rm pr}}
\renewcommand{\div}{{\rm div}}

\newtheorem{theorem}{Theorem}
\newtheorem{lemma}{Lemma}
\newtheorem{proposition}{Proposition}
\newtheorem{corollary}{Corollary}

\theoremstyle{definition}
\newtheorem{definition}{Definition}
\newtheorem{remark}{Remark}
\newtheorem{example}{Example}

\def\Xint#1{\mathchoice
{\XXint\displaystyle\textstyle{#1}}
{\XXint\textstyle\scriptstyle{#1}}
{\XXint\scriptstyle\scriptscriptstyle{#1}}
{\XXint\scriptscriptstyle\scriptscriptstyle{#1}}
\!\int}
\def\XXint#1#2#3{{\setbox0=\hbox{$#1{#2#3}{\int}$ }
\vcenter{\hbox{$#2#3$ }}\kern-.57\wd0}}

\def\dashint{\Xint-}
\numberwithin{equation}{section}

\title{On decay of almost periodic viscosity solutions to Hamilton-Jacobi equations}
\author{Evgeny Yu. Panov}
\begin{document}
\maketitle
\begin{abstract}
We establish that a viscosity solution to a multidimensional Hamilton-Jacobi equation with a convex non-degenerate hamiltonian and Bohr almost periodic initial data decays to its infimum as time $t\to+\infty$.
\end{abstract}

\section{Introduction}
In the half-space $\Pi=\R_+\times\R^n$, $\R_+=(0,+\infty)$, we consider the Cauchy problem for a
first order Hamilton-Jacobi equation
\begin{equation}\label{1}
u_t+H(\nabla_x u)=0
\end{equation}
with merely continuous hamiltonian function $H(v)\in C(\R^n)$, and with initial condition
\begin{equation}\label{2}
u(0,x)=u_0(x)\in BUC(\R^n),
\end{equation}
where $BUC(\R^n)$ denotes the Banach space of bounded uniformly continuous functions on $\R^n$
equipped with the uniform norm $\|u\|_\infty=\sup|u(x)|$. We recall the notions of
superdifferential $D^+u$ and subdifferential $D^-u$ of a continuous function $u(t,x)\in C(\Pi)$:
\begin{eqnarray*}
D^+u(t_0,x_0)=\{ \ (s,v)=\nabla \varphi(t_0,x_0) \ | \ \varphi(t,x)\in C^1(\Pi), \\ (t_0,x_0) \mbox{ is a point of local maximum of } u-\varphi \ \}, \\
D^-u(t_0,x_0)=\{ \ (s,v)=\nabla \varphi(t_0,x_0) \ | \ \varphi(t,x)\in C^1(\Pi), \\ (t_0,x_0) \mbox{ is a point of local minimum of } u-\varphi \ \}.
\end{eqnarray*}
Let us denote by $BUC_{loc}(\bar\Pi)$ the space of continuous functions on $\bar\Pi=\Cl\Pi=[0,+\infty)\times\R^n$, which are bounded and uniformly continuous in any layer $[0,T)\times\R^n$, $T>0$.
Recall the notion of viscosity solution of (\ref{1}), (\ref{2}).

\begin{definition}\label{def1}

A function $u(t,x)\in BUC_{loc}(\bar\Pi)$ is called a viscosity subsolution (v.subs. for short) of problem (\ref{1}), (\ref{2}) if
$u(0,x)\le u_0(x)$ and $s+H(v)\le 0$ for all $(s,v)\in D^+u(t,x)$, $(t,x)\in\Pi$.

A function  $u(t,x)\in BUC_{loc}(\bar\Pi)$ is called a viscosity supersolution (v.supers.) of problem (\ref{1}), (\ref{2}) if
$u(0,x)\ge u_0(x)$ and $s+H(v)\ge 0$ for all $(s,v)\in D^-u(t,x)$, $(t,x)\in\Pi$.

Finally, $u(t,x)\in BUC_{loc}(\bar\Pi)$ is called a viscosity solution (v.s.) of (\ref{1}), (\ref{2}) if it is a v.subs. and a v.supers.
of this problem simultaneously.
\end{definition}
The theory of v.s. was developed in \cite{CL,CEL}. This theory extended the earlier results of S.N. Kruzhkov \cite{KrHJ,KrHJ1}.

It is known that for each $u_0(x)\in BUC(\R^n)$ there exists a unique v.s. of problem (\ref{1}), (\ref{2}). The uniqueness readily follows from the more general comparison principle.

\begin{theorem}[see \cite{CEL}]\label{th1}
Let $u_1(t,x), u_2(t,x)\in  BUC_{loc}(\bar\Pi)$ be a v.subs. and a v.supers. of (\ref{1}), (\ref{2}) with initial data $u_{10}(x), u_{20}(x)$, respectively. Assume that $u_{10}(x)\le u_{20}(x)$ $\forall x\in\R^n$.
Then $u_1(t,x)\le u_2(t,x)$ $\forall (t,x)\in\Pi$.
\end{theorem}

\begin{corollary}\label{cor1}
Let $u_1(t,x), u_2(t,x)\in  BUC_{loc}(\bar\Pi)$ be v.s. of (\ref{1}), (\ref{2}) with initial data $u_{10}(x), u_{20}(x)$, respectively. Then for all $t>0$
$$
\inf (u_{10}(x)-u_{20}(x))\le u_1(t,x)-u_2(t,x)\le \sup (u_{10}(x)-u_{20}(x)).
$$
In particular, $\|u_1-u_2\|_\infty\le\|u_{10}-u_{20}\|_\infty$.

\end{corollary}

\begin{proof}
Let
$$
a=\inf (u_{10}(x)-u_{20}(x)), \quad b=\sup (u_{10}(x)-u_{20}(x)),
$$
Obviously, the functions $a+u_2(t,x)$, $b+u_2(t,x)$ a v.s. of (\ref{1}), (\ref{2}) with initial data $a+u_{20}(x)$, $b+u_{20}(x)$, respectively. Since $a+u_{20}(x)\le u_{10}(x)\le b+u_{20}(x)$, then by Theorem~\ref{th1}
$a+u_2(t,x)\le u_1(t,x)\le b+u_2(t,x)$ $\forall (t,x)\in\Pi$, which completes the proof.
\end{proof}

In the case when $H(0)=0$ constants are v.s. of (\ref{1}). By Corollary~\ref{cor1} with $u_1=u$, $u_2=0$ we find that a v.s. $u=u(t,x)$ is bounded, namely $\|u\|_\infty\le\|u_0\|_\infty$. Notice that the requirement $H(0)=0$ does not reduce the generality because we can make the change
$\tilde u\to u+H(0)t$, which provides a v.s. $\tilde u$ of the problem
$\tilde u_t+H(\nabla_x \tilde u)-H(0)=0$, $\tilde u(0,x)=u_0(x)$,
with new hamiltonian  $\tilde H(v)=H(v)-H(0)$ satisfying the desired condition $\tilde H(0)=0$.

We are going to study long time decay property of v.s. to the problem (\ref{1}), (\ref{2}) with almost periodic initial data. Recall that the space $AP(\R^n)$ of Bohr (or uniform) almost periodic functions is a closure of trigonometric polynomials,
i.e. finite sums $\sum a_\lambda e^{2\pi i\lambda\cdot x}$, in the space $BUC(\R^n)$ (by $\cdot$ we denote the inner product in $\R^n$). It is clear that $AP(\R^n)$ contains continuous periodic functions (with arbitrary lattice of periods). 
Let $C_R$ be the cube
$$\{ \ x=(x_1,\ldots,x_n)\in\R^n \ | \ |x|_\infty=\max_{i=1,\ldots,n}|x_i|\le R/2 \ \}, \quad R>0.$$
It is known (see for instance \cite{LevZh}~) that for each function $u\in AP(\R^n)$ there exists the mean value
$$\bar u=\dashint_{\R^n} u(x)dx\doteq\lim\limits_{R\to+\infty}R^{-n}\int_{C_R} u(x)dx$$ and, more generally, the Bohr-Fourier coefficients
$$
a_\lambda=\dashint_{\R^n} u(x)e^{-2\pi i\lambda\cdot x}dx, \quad\lambda\in\R^n.
$$
The set
$$ Sp(u)=\{ \ \lambda\in\R^n \ | \ a_\lambda\not=0 \ \} $$ is called the spectrum of an almost periodic function  $u(x)$. It is known \cite{LevZh}, that the spectrum $Sp(u)$ is at most countable.

Now we assume that the initial function $u_0(x)\in AP(\R^n)$. Let $M_0$ be the smallest additive subgroup of $\R^n$ containing $Sp(u_0)$. Notice that in the case when $u_0$ is a continuous periodic function $M_0$ coincides with the dual lattice to the lattice of periods.

We are going to find an exact condition on the hamiltonian $H(v)$ for the fulfillment of the decay property
\begin{equation}\label{dec}
u(t,x)\rightrightarrows c=\const \quad \mbox{ as } t\to+\infty.
\end{equation}
We assume that $H(0)=0$ and introduce the following non-degeneracy condition of $H(v)$ at point $v=0$ in ``resonant'' directions $\xi\in M_0$
\begin{eqnarray}\label{ND}
\forall\xi\in M_0, \xi\not=0 \ \mbox{ the functions } \ s\to H(s\xi) \nonumber\\ \mbox{ are not linear in any interval } |s|<\delta, \delta>0.
\end{eqnarray}
Notice that the similar condition (non-linearity of resonant components of the flux vector $f(u)$) is known to be an exact condition for decay of almost periodic entropy solutions to scalar conservation laws $u_t+\div_x f(u)=0$, see
\cite{PaJHDE1} and, in the periodic case, \cite{Daf,PaAIHP,PaNHM}.

Let us demonstrate that requirement (\ref{ND}) is necessary for the decay of v.s. of (\ref{1}), (\ref{2}) such that $Sp(u_0)\subset M_0$. In fact, if (\ref{ND}) fails then there exist a nonzero vector $\xi\in M_0$ and a positive $\delta>0$ such that $H(s\xi)=\alpha s$ for $|s|\le \delta$ for some $\alpha\in\R$. Then the function $u(t,x)=\frac{\delta}{2\pi}\sin(2\pi(\xi\cdot x-\alpha t))$ is a classical (and, therefore, a v.s.) of (\ref{1}), (\ref{2})
with initial function $u_0(x)=\frac{\delta}{2\pi}\sin(2\pi\xi\cdot x)$ because
$$
\nabla_x u(t,x)=s\xi, \ u_t(t,x)=-s\alpha, \quad s=\delta\cos(2\pi(\xi\cdot x-\alpha t))\in [-\delta,\delta],
$$
so that $u_t+H(\nabla_x u)=0$. Obviously, $u_0(x)$ is a periodic function and
$$
Sp(u_0)=\{ \ k\xi \ | \ k\in\Z, k\not=0 \}\subset M_0.
$$
But the decay property is evidently violated.

\medskip
In the case $n=1$ condition (\ref{ND}) reads that  $H(v)$ is not linear in any vicinity of zero. In recent preprint \cite{prep} this condition was shown to be sufficient for the decay property.

In this paper we prove the similar result in arbitrary dimension $n\ge 1$ but for a convex hamiltonian. Our main results is the following

\begin{theorem}\label{thM}
Assume that the hamiltonian $H(v)$ is convex, $H(0)=0$, and condition (\ref{ND}) is satisfied.
Then the decay property (\ref{dec}) holds. Moreover, the limit constant $c=\inf u_0(x)$.
\end{theorem}
Notice that in the case of strictly convex hamiltonian condition (\ref{ND}) is always satisfied. In this case the statement of Theorem~\ref{thM} follows from the general results of \cite[Theorem~8]{KrHJ}, \cite[Theorem~6]{KrHJ1}, for arbitrary
$u_0(x)\in BUC(\R^n)$.

We also remark that in the case of arbitrary $H(0)$ decay property (\ref{dec}) should be revised as
$$
u(t,x)+H(0)t\rightrightarrows c=\const \quad \mbox{ as } t\to+\infty.
$$

\section{The case of periodic initial function}

In the case of convex hamiltonian the unique v.s. $u(t,x)$ of (\ref{1}), (\ref{2}) can be found by the known Hopf-Lax-Oleinik formula \cite{KrHJ1,BE}
\begin{equation}\label{HLO}
u(t,x)=\min_{y\in\R^n}[u_0(y)+tH^*((x-y)/t)],
\end{equation}
where
\begin{equation}\label{Leg}
H^*(p)=\sup_{v\in\R^n} [p\cdot v-H(v)]
\end{equation}
is the Legendre transform of $H$. To simplify the proofs, we will assume in addition that the following coercivity condition is satisfied (in Remark~\ref{rem1} below we demonstrate how to avoid this condition)
\begin{equation}\label{coerc}
H(v)/|v|\to +\infty \quad \mbox{ as } v\to\infty
\end{equation}
(here and in the sequel we denote by $|v|$ the Euclidean norm of a finite-dimensional vector $v$).
Under this condition the Legendre conjugate function $H^*(p)$ is everywhere defined convex function satisfying the coercivity condition: $H^*(p)/|p|\to +\infty$ as $p\to\infty$.

It is known that for every $v_0\in\R^n$ the sub-differential $D^-H(v_0)$ of a convex function $H(v)$ on $\R^n$ coincides with the set
$$
\partial H(v_0)\doteq \{ \ p\in\R^n \ | \ H(v)-H(v_0)\ge p\cdot (v-v_0) \ \},
$$
which is a nonempty convex compact in $\R^n$. By the assumption $H(0)=0$ the conjugate function $H^*(p)\ge 0$ and $H^*(p_0)=0$ if and only if $p_0\in\partial H(0)$. We fix such
$p_0$ and introduce the convex set $\partial H^*(p_0)$. Since $0=H^*(p_0)=\min H(p)$ then $0\in\partial H^*(p_0)$. By the duality,
$$
H(v)=H^{**}(v)=\max_{p\in\R^n}[p\cdot v-H^*(p)].
$$
As readily follows from this relation,
\begin{equation}\label{lin}
H(v)=p_0\cdot v \quad \forall v\in\partial H^*(p_0).
\end{equation}

We fix $\varepsilon>0$ and introduce the polar set
\begin{equation}\label{dual}
G=(\partial H^*(p_0))'=\{ \ p\in\R^n \ | \ p\cdot v\le\varepsilon \ \forall v\in\partial H^*(p_0) \ \}.
\end{equation}
Then $G$ is a closed convex set and, by the bipolar theorem \cite[Theorem 14.5]{Rock},
\begin{equation}\label{ddual}
\partial H^*(p_0)=G'=\{ \ v\in\R^n \ | \ p\cdot v\le\varepsilon \ \forall p\in G \ \}.
\end{equation}
Let $\T^n=\R^n/\Z^n$ be the torus, $\pr:\R^n\to\T^n$ be the natural projection.

\begin{proposition}\label{prop1}
Let $G\subset\R^n$ be a convex set such that $\pr(G)$ is not dense in $\T^n$. Then there exists $\xi\in\Z^n$ such that the linear functional $p\to\xi\cdot p$ is bounded on $G$.
\end{proposition}

\begin{proof}
Without loss of generality we may suppose that $G$ is a closed convex set and that $0\in G$.
We define the dimension $\dim A$ of any set $A\subset\R^n$ as the dimension of its linear span.
Let $m(G)$ be the maximal of such integer $m$ that there exists a convex cone $C\subset G$ of dimension $m$. Since the trivial cone $\{0\}\subset G$, then $0\le m(G)\le n$.
We will prove our statement by induction in the value $k=n-m(G)$. If $k=0$ then there exists a cone $C\subset G$ of full dimension $n$. This implies that $C$ contains some ball $B_\delta(x_0)=\{ x\in\R^n \ | \ |x-x_0|\le\delta \ \}$. Obviously, $B_{r\delta}(rx_0)=rB_\delta(x_0)\subset C$ for any $r>0$. In particular, the set $G\supset C$ contains balls of arbitrary  radius. This implies that $G$ contains a cube $K$ with side greater than $1$. Since $\pr(K)=\T^n$, we conclude that $\pr(G)=\T^n$. Hence in the case $k=0$ the assumptions that $\pr(G)$ is not dense in $\T^n$ cannot be satisfied and our assertion is true.

Now assume that $k>0$ and the statement of our proposition holds for all sets $G$ such that $n-m(G)<k$. We have to prove our statement in the case $n-m(G)=k$. Let $C\subset G$ be a cone of maximal dimension $n-k$, $L$ be the linear span of $C$. By our assumption $K=\Cl\pr(C)$ is a proper compact subset of $\T^n$. Introduce the set
$$
M=\{ \ x\in\R^n \ | \ \pr(\lambda x)+y\in K \ \forall y\in K, \ \lambda\ge 0 \ \}.
$$
Here $+$ is a standard group operation on the torus $\T^n$.
Since $$\pr(\lambda (c_1x_1+c_2x_2))+y=\pr(\lambda c_1x_1)+(\pr(\lambda c_2x_2)+y)\in K$$ for all $y\in K$, then $c_1x_1+c_2x_2\in M$ whenever $x_1,x_2\in M$, $c_1,c_2\ge 0$.
This means that $M$ is a convex cone. Let us demonstrate that $-x\in M$ for each $x\in M$. For that we define the standard metric $d$ on $\T^n$,
$$
d(y_1,y_2)=\min\{ \ |x_1-x_2| \ | \ y_i=\pr(x_i), i=1,2 \ \},
$$
and introduce for $x_0\in\R^n$, $y\in K$ the function
$f(s)=d(\pr(sx_0)+y,K)=\min\limits_{z\in K} d(\pr(sx_0)+y,z)$, $s\in\R$. It is clear that $f(s)$ is an almost periodic function (as a composition of the continuous periodic function
$g(x)=d(\pr(x)+y,K)$ and the linear embedding $s\to sx_0$~). If $x_0\in M$ then $f(s)=0$ for all $s\ge 0$. Since $f(s)$ is an almost periodic function, the latter is possible only if $f\equiv 0$. In particular, $\pr(sx_0)+y\in K$ for all $s\le 0$ and $y\in K$, that is, $-x_0\in M$. Hence, $M$ is a symmetric convex cone, i.e., it is a linear space. If $x\in C$, $y=\pr(z)$, $z\in C$, then $\pr(\lambda x)+y=\pr(\lambda x+z)\in \pr(C)$. Thus, $\pr(\lambda x)+y\in\pr(C)$ whenever $y\in\pr(C)$, $\lambda\ge 0$. By the continuity of the group operation we find that $\pr(\lambda x)+y\in K$ for all $y\in K=\Cl\pr(C)$, $\lambda\ge 0$, that is, $x\in M$. We conclude that $C\subset M$. Since $M$ is a linear subspace, it follows that $L\subset M$. Notice that $0\in K$, $\pr(M)=\pr(M)+0\subset K\not=\T^n$. Let $S=\Cl\pr(M)\subset K$. Then $S$ is a proper closed subgroup of $\T^n$. By Pontryagin duality \cite{Pont}, there exists a character $\chi(y)=e^{2\pi i\xi\cdot y}$, $\xi\in\Z^n$, $\xi\not=0$, such that $\chi(y)=1$ on $S$. It follows that $\xi\cdot x=0$ for all $x\in M$, and in particular, $\xi\in L^\perp$. If the linear functional
$x\to\xi\cdot x$ is bounded on $G$, then the desired statement is proved. Assuming the contrary, that is, this functional is not bounded on $G$, we can find the sequence $p_r\in G$, $r\in\N$, such that $\xi\cdot p_r\to\infty$ as $r\to\infty$.   We introduce the new convex set $G'=\Cl(L+G)$ and remark that
\begin{eqnarray}\label{conv1}
\pr(L+G)=\pr(L)+\pr(G)\subset K+\pr(G)=\nonumber\\ \Cl\pr(C)+\pr(G)\subset \Cl\pr(C+G)\subset\Cl\pr(G).
\end{eqnarray}
We utilized that $C+G\subset G$. To prove this inclusion, we fix $x\in C$, $p\in G$ and observe that
\begin{equation}\label{conv2}
x+p=\lim_{r\to\infty}\left(x+\frac{r-1}{r}p\right),
\end{equation}
while $x+\frac{r-1}{r}p=\frac{1}{r}rx+\frac{r-1}{r}p\in G$ for all $r>1$ by the convexity of $G$ (notice that $rx\in C\subset G$). Since $G$ is closed, relation (\ref{conv2}) implies that $x+p\in G$ for each $x\in C$, $p\in G$, as was to be proved.

By (\ref{conv1}) $\Cl\pr(G')=\Cl\pr(G)$ is a proper subset of $\T^n$. Let $q_r\in L$ be orthogonal projection of $p_r$, so that $p_r-q_r\perp L$. Since $\xi\cdot(p_r-q_r)=\xi\cdot p_r\to\infty$ as $r\to\infty$, we have $\alpha_r=|p_r-q_r|\to +\infty$ as $r\to\infty$. Since $0,p_r-q_r\in G'$ then $\lambda(\alpha_r)^{-1}(p_r-q_r)\in G'$ if $\alpha_r>\lambda\ge 0$. Passing to a subsequence if necessary, we can suppose that the sequence of unite vectors $(\alpha_r)^{-1}(p_r-q_r)\to h$ as $r\to\infty$. Evidently, $|h|=1$ and $h\perp L$. Besides, $\lambda h=\lim_{r\to\infty}\lambda(\alpha_r)^{-1}(p_r-q_r)\in G'$ because the set $G'$ is closed. We find that the cone $C'=L+\{\ \lambda h \ | \ \lambda\ge 0 \ \}\subset G'$ while $\dim C'=m(G)+1$. We see that $n-m(G')<k$. By the induction hypothesis there exists a vector $\xi\in\Z^n$, $\xi\not=0$ such that the corresponding linear functional $x\to\xi\cdot x$ is bounded on $G'$ and therefore also on $G\subset G'$. We prove the assertion of our proposition for the case $n-m(G)=k$. By the principle of mathematical induction, this completes the proof.
\end{proof}

\begin{corollary}\label{cor2}
Assume that the following non-degeneracy condition holds:
\begin{eqnarray}\label{ND1}
\forall\xi\in\Z^n, \xi\not=0 \mbox{ the functions } s\to H(s\xi) \nonumber\\
\mbox{ are not linear in any vicinity of zero. }
\end{eqnarray}
Then for each $\varepsilon>0$ the set $\pr(G)$ is dense in $\T^n$, where the convex set $G$ is given by (\ref{dual}).
\end{corollary}

\begin{proof}
Assuming the contrary and applying Proposition~\ref{prop1}, we can find $\xi\in\Z^n$, $\xi\not=0$, and a positive constant $c$ such that
$|\xi\cdot p|\le c$ for all $p\in G$. In view of (\ref{ddual}) this implies that $s\xi\in G'=\partial H^*(p_0)$ for $|s|<\delta=\varepsilon/c$. By (\ref{lin}) we find that the function $H(s\xi)=sp_0\cdot\xi$ is linear on the interval $|s|<\delta$. But this contradicts to our assumption. Thus, $\pr(G)$ is dense in $\T^n$. The proof is complete.
\end{proof}

Let $u_0(x)\in C(\T^n)$ be a periodic function (with the standard lattice of periods $\Z^n$), and $u=u(t,x)$ be the unique v.s. of the problem (\ref{1}), (\ref{2}). Observe that the group $M_0$ coincides with $\Z^n$, and condition (\ref{ND}) reduces to (\ref{ND1}).
Now, we are ready to prove our main result.

\begin{theorem}\label{th2}
Under non-degeneracy condition (\ref{ND1}) a v.s. $u(t,x)$ of (\ref{1}), (\ref{2}) satisfies the following decay property:
\begin{equation}\label{dec1}
u(t,x)\rightrightarrows c=\min u_0(x) \quad \mbox{ as } t\to+\infty.
\end{equation}
\end{theorem}

\begin{proof}
We fix $p_0\in\partial H(0)$, $\varepsilon>0$ and consider the corresponding set $G$ introduced in (\ref{dual}). Since $u_0(y)$ is an uniformly continuous function on the compact $\T^n$, there exists such $\delta>0$ that
\begin{equation}\label{uc}
|u_0(y_1)-u_0(y_2)|<\varepsilon \quad \forall y_1,y_2\in\T^n, d(y_1,y_2)\le\delta.
\end{equation}
By Corollary~\ref{cor2}, the set $\pr(G)$ is dense in $\T^n$. Therefore, there exists a finite $\delta$-net $Y=\{y_1,\ldots,y_m\}\subset\pr(G)$. We choose $q_k\in G$ such that $y_k=\pr(q_k)$, $k=1,\ldots,m$. Let a point $y_*\in\T^n$ be such that $u_0(y_*)=c=\min u_0(y)$. Since $Y$ is a $\delta$-net in $\T^n$, then for each $(t,x)\in\Pi$ there exists such $k\in\{1,\ldots,m\}$ that
$$
d(\pr(x-p_0t-q_k),y_*)=d(\pr(x-p_0t)-y_*,y_k)\le\delta.
$$
In view of (\ref{uc}), $u_0(x-p_0t-q_k)<u_0(y_*)+\varepsilon=c+\varepsilon$. From (\ref{HLO}) it follows that
\begin{eqnarray}\label{conv3}
u(t,x)\le u_0(x-p_0t-q_k)+tH^*((x-(x-p_0t-q_k))/t)=\nonumber\\ u_0(x-p_0t-q_k)+tH^*(p_0+q_k/t)<c+\varepsilon+tH^*(p_0+q_k/t).
\end{eqnarray}
Notice also that in view of Corollary~\ref{cor1} with $u_1=u(t,x)$, $u_2\equiv 0$, $u(t,x)\ge c$ for all $(t,x)\in\Pi$.
From this inequality and (\ref{conv3}) it now follows that
\begin{equation}\label{conv4}
c\le u(t,x)<c+\varepsilon+\alpha(t),
\end{equation}
where
$$
\alpha(t)=\max_{k=1,\ldots,m} tH^*(p_0+q_k/t).
$$
Since $H^*(p)$ is a convex function and $H^*(p_0)=0$, there exist limits
$$
\lim_{t\to+\infty}tH^*(p_0+q_k/t),
$$
which coincide with directional derivatives $D_{q_k}H^*(p_0)$. It is known (see \cite{Rock}) that
$$
D_{q_k}H^*(p_0)=\max_{v\in\partial H^*(p_0)} q_k\cdot v,
$$
and since $q_k\in G$ we see that $D_{q_k}H^*(p_0)\le\varepsilon$ for all $k=1,\ldots,m$.
Therefore, $\lim\limits_{t\to+\infty}\alpha(t)\le\varepsilon$ and it follows from (\ref{conv4}) that
$$
\limsup_{t\to +\infty}\|u(t,\cdot)-c\|_\infty\le 2\varepsilon.
$$
To complete the proof it only remains to notice that $\varepsilon>0$ is arbitrary.
\end{proof}

\begin{remark}\label{rem1}
The statement of Theorem~\ref{2} remains valid for arbitrary convex hamiltonian $H(v)$, which may not satisfy the coercivity condition (\ref{coerc}).

Assume first that the initial function $u_0(x)$ is Lipschitz:
$$|u_0(x)-u_0(y)|\le L|x-y| \quad \forall x,y\in\R^n,$$ $L>0$ is a Lipschitz constant.
By Corollary~\ref{cor1} we have
$$
|u(t,x+h)-u(t,x)|\le \sup |u_0(x+h)-u_0(x)|\le L|h| \quad \forall x,h\in\R^n, t>0.
$$
Thus, the functions $u(t,\cdot)$ satisfy the Lipschitz condition with the constant $L$. Therefore, the generalized gradient $\nabla_x u\in L^\infty(\Pi,\R^n)$, $\|\nabla_x u\|_\infty\le L$. This readily implies that $|v|\le L$ whenever $(s,v)\in D^\pm u(t,x)$, $(t,x)\in\Pi$.
We see that the behavior of $H(v)$ for $|v|>L$ does not matter and we can always improve the convex hamiltonian $H(v)$ in the domain $|v|>L$ in such a way that the corrected hamiltonian satisfies the coercivity assumption. It is clear that the non-degeneracy condition (\ref{ND1}) remains valid. By Theorem~\ref{th2} we conclude that decay property (\ref{dec1}) holds.

In the general case we construct the sequence $u_{0k}\in C(\T^n)$, $k\in\N$, of periodic Lipschitz functions such that $u_{0k}\to u_0$ as $k\to\infty$ in $C(\T^n)$. Let $u_k=u_k(t,x)$ be a v.s. of (\ref{1}), (\ref{2}) with initial data $u_{0k}$.
Then, taking into account Corollary~\ref{cor1}, we find that as $k\to\infty$
\begin{equation}\label{conv5}
u_k(t,x)\rightrightarrows u(t,x), \quad c_k=\min u_{0k}(y)\to c=\min u_0(y).
\end{equation}
As was already proved, for each $k\in\N$
$$
u_k(t,\cdot)\rightrightarrows c_k \quad \mbox{ as } t\to+\infty.
$$
In view of (\ref{conv5})
we can pass to the limit as $k\to\infty$ in the above relation and derive the desired result
(\ref{dec1}).
\end{remark}

We underline that condition (\ref{ND1}) can be satisfied even for a hamiltonian linear on each of two half-spaces with a common boundary hyper-space.
\begin{example}
Let $p\in\R^n$ be a nonzero vector. We consider the equation
$$
u_t+|\partial_p u|=0,
$$
where $\partial_p u=p\cdot\nabla_x u$ is the directional derivative of $u$. Obviously, the hamiltonian $H(v)=|p\cdot v|$ satisfies (\ref{ND1}) if and only if
$p\cdot\xi\not=0$ for every $\xi\in\Z^n$, $\xi\not=0$. This means that the coordinates $p_j$, $j=1,\ldots,n$, of the vector $p$ are linearly independent over the field $\Q$ of rationals.
\end{example}

\section{The case of almost periodic initial data}
In this section we prove Theorem~\ref{thM} in the general case $u_0(x)\in AP(\R^n)$.

We will need some simple general properties of v.s. collected in the following lemma.

\begin{lemma}\label{lem1}
(i) If $u(t,x)$ is a v.s. of (\ref{1}), (\ref{2}), then $v=-u(t,x)$ is a v.s. to the problem
$$
v_t-H(-\nabla_x v)=0, \quad v(0,x)=-u_0(x);
$$

(ii) Let $y=Ax$ be a non-degenerate linear operator on $\R^n$, $v_0(y)\in BUC(\R^n)$, $v(t,y)\in BUC_{loc}(\Pi)$. Then the function $u(t,x)=v(t,Ax)$ is a v.s. of (\ref{1}), (\ref{2}) with initial data $u_0(x)=v_0(Ax)$ if and only if $v(t,y)$ is a v.s. of the problem
$$
v_t+H(A^*\nabla_y v)=0, \quad v(0,y)=v_0(y),
$$
where $A^*$ is the conjugate operator;

(iii) Let $H(p,q)\in C(\R^n\times\R^m)$. We consider the equation
\begin{equation}\label{a1}
U_t+H(\nabla_x U,\nabla_y U)=0
\end{equation}
in the half-space $\{ \ (t,x,y) \ | \ t>0, x\in\R^n, y\in\R^m \ \}$. Then $U(t,x,y)=u(t,y)$ is a non-depending on $x$ v.s. of (\ref{a1}) if and only if $u(t,y)$ is a v.s. of the reduced equation
$$
u_t+H(0,\nabla_y u)=0, \quad (t,y)\in\R_+\times\R^m.
$$
\end{lemma}
\begin{proof}
(i) As is easy to verify, $(s,w)\in D^\pm v(t_0,x_0)$ if and only if $(-s,-w)\in D^\mp u(t_0,x_0)$. Since $u(t,x)$ is a v.s. of (\ref{1}) we obtain that, respectively, $\pm(-s+H(-w))\ge 0$, i.e., $\mp(s-H(-w))\ge 0$. By the definition, this means that $v(t,x)$ is a v.s. of the equation $v_t-H(-\nabla_x v)=0$. Since the initial condition $v(0,x)=-u_0(x)$ is evident, this completes the proof of (i).

Assertion (ii) follows from the fact that $(t_0,x_0)$ is a point of local maximum (minimum) of $u(t,x)-\psi(t,Ax)$, with $\psi(t,y)\in C^1(\Pi)$,
if and only if $(t_0,Ax_0)$ is a point of local maximum (minimum) of $v(t,y)-\psi(t,y)$ and from the classical identity $A^*\nabla_y \psi(t,y)=\nabla_x \psi(t,Ax)$, $y=Ax$.

Finally, assertion (iii) readily follows from the evident equalities
$$D^\pm U(t,x,y)=\{ (s,0,v)\in \R\times\R^n\times\R^m \ | \ (s,v)\in D^\pm u(t,y) \ \}.$$
\end{proof}

We also will use one more property of v.s.
In the half space $\R_+\times\R^n\times\R^m$ we consider the Cauchy problem
for equation
\begin{equation}\label{l1}
u_t+H(\nabla_x u)=0, \quad u=u(t,x,y), \ (t,x,y)\in \R_+\times\R^n\times\R^m,
\end{equation}
with initial condition
\begin{equation}\label{l2}
u(0,x,y)=u_0(x,y)\in BUC(\R^n\times\R^m).
\end{equation}

\begin{lemma}\label{lem2}
A function $u(t,x,y)\in BUC_{loc}(\R_+\times\R^n\times\R^m)$ is a v.s. of (\ref{l1}),
(\ref{l2}) if and only if for all fixed $y\in\R^m$ the functions $u^y(t,x)=u(t,x,y)$
is a v.s. of (\ref{1}), (\ref{2}) with initial data $u_0^y(x)=u_0(x,y)$.
\end{lemma}
For the sake of completeness we provide below the proof of this lemma.
\begin{proof}
Let $u(t,x,y)$ be a v.s. of (\ref{l1}), (\ref{l2}), and $y_0\in\R^m$. We assume that $\varphi(t,x)\in C^1(\Pi)$ and
$(t_0,x_0)\in\Pi$ is a point of local maximum of $u^{y_0}-\varphi$. Moreover, replacing $\varphi$ by $\varphi(t,x)+(t-t_0)^2+|x-x_0|^2+u(t_0,x_0,y_0)-\varphi(t_0,x_0)$, we can suppose, without loss of generality, that $(t_0,x_0)\in\Pi$ is a point of strict local maximum of $u^{y_0}-\varphi$, and that in this point
$u^{y_0}(t_0,x_0)-\varphi(t_0,x_0)=0$.  Therefore, there exists $c>0$ such that
$$
\varphi(t,x)-u(t,x,y_0)>c \quad \forall (t,x)\in\Pi, \ (t-t_0)^2+|x-x_0|^2=r^2,
$$
for some $r\in (0,t_0)$.
By the continuity there exists $h>0$ such that $\varphi(t,x)-u(t,x,y)>c/2$ for all $(t,x,y)\in\R_+\times\R^n\times\R^m$, $(t-t_0)^2+|x-x_0|^2=r^2$, $|y-y_0|\le h$. We can choose such $C_0>0$ that
$$
C_0h^2-c>\max\{ \ u(t,x,y)-\varphi(t,x) \ | \ (t-t_0)^2+|x-x_0|^2\le r^2, \ |y-y_0|=h \ \}.
$$
Then for each natural $k>C_0$ the function $p_k(t,x,y)=\varphi(t,x)+k|y-y_0|^2\in C^1(\R_+\times\R^n\times\R^m)$ and satisfies the property
\begin{equation}\label{l3}
p_k(t,x,y)-u(t,x,y)>c/2>0=p_k(y_0,x_0,y_0)-u(t_0,x_0,y_0)
\end{equation}
$\forall (t,x,y)\in \partial V_{r,h}$, where we denote by $V_{rh}$ the domain
$$
V_{rh}=\{ \ (t,x,y)\in \R_+\times\R^n\times\R^m \ | \ (t-t_0)^2+|x-x_0|^2<r^2, \ |y-y_0|<h \ \}.
$$
In view of (\ref{l3}) the point $(t_k,x_k,y_k)$ such that
$$
p_k(t_k,x_k,y_k)-u(t_k,x_k,y_k)=\min\limits_{(t,x,y)\in \Cl V_{rh}} (p_k(t,x,y)-u(t,x,y))
$$
lies in $V_{rh}$ and, therefore, it is a point of local maximum of the difference $u(t,x,y)-p_k(t,x,y)$.
Since $\nabla p_k(t,x,y)=(\partial_t\varphi(t,x),\nabla_x\varphi(t,x),2k(y-y_0))$, then by the definition of v.s. of
(\ref{l1})
\begin{equation}\label{l4}
\partial_t\varphi(t_k,x_k)+H(\nabla_x\varphi(t_k,x_k))\le 0.
\end{equation}
Since $\min\limits_{(t,x,y)\in \Cl V_{rh}} (p_k(t,x,y)-u(t,x,y))\le p_k(t_0,x_0,y_0)-u(t_0,x_0,y_0)=0$, then  $k|y_k-y_0|^2\le m=\max\limits_{(t,x,y)\in \Cl V_{rh}} (u(t,x,y)-\varphi(t,x))$. In particular $y_k\to y_0$ as $k\to\infty$. Taking into account that $(t_0,x_0)$ is a point of strict local maximum of $u(t,x,y_0)-\varphi(t,x)$,
we derive that $(t_k,x_k)\to (t_0,x_0)$ as $k\to\infty$. Therefore, it follows from (\ref{l4}) in the limit as $k\to\infty$ that
$$
\partial_t\varphi(t_0,x_0)+H(\nabla_x\varphi(t_0,x_0))\le 0.
$$
This means that $u(t,x,y_0)$ is a v.subs. of (\ref{1}).
By the similar reasons we obtain that
$$
\partial_t\varphi(t_0,x_0)+H(\nabla_x\varphi(t_0,x_0))\ge 0
$$
whenever $(t_0,x_0)$ is a point of strict local minimum of $u(t,x,y_0)-\varphi(t,x)$, where $\varphi(t,x)\in C^1(\Pi)$, that is, $u(t,x,y_0)$ is a v.supers. of (\ref{1}). Thus, $u(t,x,y_0)$ is a v.s. of (\ref{1}) for each $y_0\in\R^m$.

Conversely, assume that $u^y(t,x)$ is a v.s. of (\ref{1}) for every $y\in\R^m$. Suppose that $\varphi(t,x,y)\in C^1(\R_+\times\R^n\times\R^m)$ and that $(t_0,x_0,y_0)$ is a point of local maximum (minimum) of $u(t,x,y)-\varphi(t,x,y)$. Then the point $(t_0,x_0)\in\Pi$ is a point of local maximum (minimum) of
$u^{y_0}(t,x)-\varphi(t,x,y_0)$. Since $u^{y_0}$ is a v.s. of (\ref{1}) then
$\varphi_t(t_0,x_0,y_0)+H(\nabla_x\varphi(t_0,x_0,y_0))\le 0$ (respectively, \\ $\varphi_t(t_0,x_0,y_0)+H(\nabla_x\varphi(t_0,x_0,y_0))\ge 0$). Hence, $u(t,x,y)$ is a v.s. of (\ref{l1}).
To complete the proof it only remains to notice that initial condition (\ref{l2}) is satisfied if and only if $u^y(t,x)$ satisfies (\ref{2}) with initial data $u_0^y$ for all $y\in\R^m$.
\end{proof}

Now, we can extend the statement of Lemma~\ref{lem1}(ii) to the case of arbitrary linear maps.

\begin{proposition}\label{prop2} Let $\Lambda:\R^n\to\R^m$ be a linear map, and $v(t,y)$ be a v.s. to the problem
\begin{equation}\label{1r}
v_t+H(\Lambda^*\nabla_y v)=0, \quad v(0,y)=v_0(y)
\end{equation}
in the half-space $(t,y)\in \R_+\times\R^m$. Then $u(t,x)=v(t,\Lambda x)$ is a v.s. of original problem (\ref{1}), (\ref{2})
with initial function $u_0(x)=v_0(\Lambda x)$.
\end{proposition}

\begin{proof}
We introduce the invertible linear operator $\tilde\Lambda$ on the extended space $\R^{n+m}$, defined by the equality $\tilde\Lambda (x,z)=(x,z+\Lambda x)$.
Since ${\tilde\Lambda^*(x,y)=(x+\Lambda^*y,y)}$, equation (\ref{1r}) can be rewritten in the form $$v_t+H(\tilde\Lambda^*(0,\nabla_y v))=0,$$ where $H(p,q)=H(p)$, $p\in\R^n$, $q\in\R^m$. By Lemma~\ref{lem1}(iii) the function $v=v(t,y)$ is a v.s. of equation
$$v_t+H(\tilde\Lambda^*(\nabla_x v,\nabla_y v))=0$$ in the extended domain $(t,x,y)\in\R_+\times\R^n\times\R^m$.
Then, by Lemma~\ref{lem1}(ii) the function $u(t,x,z)=v(t,z+\Lambda x)$ is a v.s. of (\ref{1}) considered in the extended domain $(t,x,z)\in\R_+\times\R^n\times\R^m$.  Applying Lemma~\ref{lem2}
we conclude that $u^z(t,x)=u(t,x,z)$ is a v.s. of (\ref{1}) for all $z\in\R^m$. Taking $z=0$ we find that
$u(t,x)=v(t,\Lambda x)$ is a v.s. of (\ref{1}). It is clear that $u(0,x)=v_0(\Lambda x)=u_0(x)$, that is, $u(t,x)$
is a v.s. of original problem (\ref{1}), (\ref{2}).
\end{proof}

Now we are ready to prove our main Theorem~\ref{thM}.

\begin{proof}[Proof of Theorem~\ref{thM}]
We first assume that the initial function is a trigonometric polynomial $\displaystyle u_0(x)=\sum_{\lambda\in S}a_\lambda e^{2\pi i\lambda\cdot x}$. Here $S=Sp(u_0)\subset\R^n$ is a finite set.  Then the subgroup $M_0$ is a finite generated torsion-free abelian group and therefore it is a free abelian group of
finite rank (see \cite{Lang}). Hence, there is a basis $\lambda_j\in M_0$, $j=1,\ldots,m$, so that every element $\lambda\in M_0$ can be uniquely represented as $\displaystyle\lambda=\lambda(\bar k)=\sum_{j=1}^m k_j\lambda_j$, $\bar k=(k_1,\ldots,k_m)\in\Z^m$. In particular, we can represent the initial function as
$$u_0(x)=\sum_{\bar k\in J} a_{\bar k}e^{2\pi i\sum_{j=1}^m k_j\lambda_j\cdot x}, \quad a_{\bar k}\doteq a_{\lambda(\bar k)},$$
where $J=\{ \ \bar k\in\Z^m \ | \ \lambda(\bar k)\in S \ \}$ is a finite set. By this representation $u_0(x)=v_0(y(x))$, where
$$
v_0(y)=\sum_{\bar k\in J} a_{\bar k}e^{2\pi i\bar k\cdot y}
$$
is a periodic function on $\R^m$ with the standard lattice of periods $\Z^m$ while $y=\Lambda x$ is a linear map from  $\R^n$ to $\R^m$ defined by the equalities $\displaystyle y_j=\lambda_j\cdot x$, $j=1,\ldots,m$. We consider the Hamilton-Jacobi equation (\ref{1r}). Let $v(t,y)$ be a v.s. of the Cauchy problem for equation (\ref{1r}) with initial function $v_0(y)$. Then by Proposition~\ref{prop2} we have the identity $u(t,x)=v(t,\Lambda x)$. Let us verify that the hamiltonian $\tilde H(w)=H(\Lambda^* w)$ of equation (\ref{1r}) satisfies condition (\ref{ND1}). Indeed,
\begin{equation}\label{conv6}
\tilde H(s\xi)=H(s\Lambda^*\xi)=H(s\lambda),
\end{equation}
where $\lambda=\Lambda^*\xi=\sum_{j=1}^m \xi_j\lambda_j\in M_0$ for each $\xi=(\xi_1,\ldots,\xi_m)\in\Z^m$. Since  $\lambda_j$, $j=1,\ldots,m$, is a basis, $\lambda=\Lambda^*\xi\not=0$ if $\Z^m\ni\xi\not=0$. By assumption (\ref{ND})
for every $\lambda\in M_0$, $\lambda\not=0$, the function $s\to H(s\lambda)$ is not linear in any vicinity of zero.
In view of (\ref{conv6}) the convex hamiltonian $\tilde H(w)$ satisfies the non-degeneracy requirement (\ref{ND1}) (with dimension $m$ instead of $n$). By Theorem~\ref{th2}
$$
v(t,y)\rightrightarrows c=\min v_0(y)
$$
as $t\to+\infty$. Since $u_0(x)=v_0(\Lambda x)$, $u(t,x)=v(t,\Lambda x)$, the latter relation reduces to the following one
$$
u(t,x)\rightrightarrows c=\min v_0(y) \ \mbox{ as } t\to +\infty.
$$
Observe, that the set $\pr(\Lambda(\R^n))$ is dense in $\T^m$ (in particular, this follows from Proposition~\ref{prop1}) while $v_0(y)\in C(\T^m)$.
Therefore, $c=\min v_0(y)=\inf v_0(\Lambda x)=\inf u_0(x)$, which completes the proof in the case when $u_0(x)$ is a trigonometric polynomial.

\medskip
The general case of arbitrary $u_0\in AP(\R^n)$ will be treated by approximation arguments. There exists a sequence of trigonometric polynomials $u_{0m}(x)$, $m\in\N$, such that $Sp(u_{0m})\subset M_0$ and $u_{0m}\rightrightarrows u_0$ as $m\to\infty$. For instance, we can choose $u_{0m}$ as the sequence of Bochner-Fej\'er trigonometric polynomials, see \cite{LevZh}. Let $u_m(t,x)$ be a v.s. of (\ref{1}), (\ref{2}) with initial data $u_{0m}$. By Corollary~\ref{cor1}
$$
\|u_m-u\|_\infty\le\|u_{0m}-u_0\|_\infty\to 0 \ \mbox{ as } t\to\infty.
$$
As we have already established in the first part of the proof, v.s. $u_m$ satisfy the decay property
\begin{equation}\label{conv7}
u_m(t,x)\rightrightarrows c_m=\inf u_{0m}(x).
\end{equation}
Since $u_m\rightrightarrows u$, $c_m\to c=\inf u_0(x)$ as $m\to\infty$, then the assertion of Theorem~\ref{thM} follows from (\ref{conv7}) in the limit as $m\to\infty$.
\end{proof}

\begin{remark}\label{rem2}
In the case of concave hamiltonian $H(v)$
\begin{equation}\label{deconc}
u(t,\cdot)\rightrightarrows c=\sup u_0(x) \quad \mbox{ as } t\to+\infty.
\end{equation}
Indeed, by Lemma~\ref{lem1}(i) the function $w=-u(t,x)$ is a v.s. of the problem
$$
w_t-H(-w_x)=0, \quad w(0,x)=-u_0(-x),
$$
with the convex hamiltonian $-H(-w)$. By Theorem~\ref{thM}
$$
w(t,x)=-u(t,x)\rightrightarrows \inf -u_0(x)=-\sup u_0(x) \quad \mbox{ as } t\to+\infty,
$$
which reduces to (\ref{deconc}).
\end{remark}

\textbf{Acknowledgments}.
The research was carried out under support of the Russian Foundation for Basic Research (grant no. 15-01-07650-a)  and the Ministry of Education and Science of Russian Federation (project no. 1.445.2016/1.4).

\end{document}